\documentclass[11pt,leqno]{amsart}
\theoremstyle{definition}
\usepackage{indentfirst, amssymb,amsmath,amsthm}
\usepackage{csquotes}
\usepackage{hyperref}
\usepackage{mathrsfs}
\newtheorem*{theoA}{Theorem A}
\newtheorem*{theoB}{Theorem B}
\newtheorem*{theoC}{Theorem C}
\newtheorem*{theoD}{Theorem D}
\newtheorem*{theoE}{Theorem E}
\newtheorem*{theoF}{Theorem F}
\newtheorem*{theoG}{Theorem G}

\newtheorem{theo}{Theorem}[section]
\newtheorem{lem}{Lemma}[section]
\newtheorem{cor}{Corollary}[section]

\newtheorem{defi}{Definition}[section]
\newtheorem{rem}{Remark}[section]
\newtheorem{question}{Question}[section]
\newcommand{\ol}{\overline}
\newcommand{\be}{\begin{equation}}
\newcommand{\ee}{\end{equation}}
\newcommand{\beas}{\begin{eqnarray*}}
\newcommand{\eeas}{\end{eqnarray*}}
\newcommand{\bea}{\begin{eqnarray}}
\newcommand{\eea}{\end{eqnarray}}

\numberwithin{equation}{section}
\begin{document}
\title[value distribution of some differential monomials]{value distribution of some differential monomials}
\date{}
\author[B. Chakraborty, et al.]{Bikash Chakraborty$^{1}$, Sudip Saha$^{2}$, Amit Kumar Pal$^{3}$ and Jayanta Kamila$^{4}$}
\date{}
\address{$^{1}$Department of Mathematics, Ramakrishna Mission Vivekananda Centenary College, Rahara,
West Bengal 700 118, India.}
\email{bikashchakraborty.math@yahoo.com, bikash@rkmvccrahara.org}
\address{$^{2}$Department of Mathematics, Ramakrishna Mission Vivekananda Centenary College, Rahara,
West Bengal 700 118, India.}
\email{sudipsaha814@gmail.com}
\address{$^{3}$Department of Mathematics, University of Kalyani, Kalyani, West Bengal 741 235, India.}
\email{mail4amitpal@gmail.com}
\address{$^{4}$Department of Mathematics, Ramakrishna Mission Vivekananda Centenary College, Rahara,
West Bengal 700 118, India.}
\email{kamilajayanta@gmail.com}
\maketitle
\let\thefootnote\relax
\begin{abstract} Let $f$ be a transcendental meromorphic function defined in the complex plane $\mathbb{C}$. We consider the value distribution of the differential polynomial $f^{q_{0}}(f^{(k)})^{q_{k}}$, where $q_{0}(\geq 2), q_{k}(\geq 1)$ are $k(\geq1)$ non-negative integers.  We obtain a quantitative estimation of the characteristic function $T(r, f)$ in terms of $\ol{N}\left(r,\frac{1}{f^{q_{_{0}}}(f^{(k)})^{q_{k}}-1}\right)$.\par
Our result generalizes the results obtained by Xu et al. (Math. Inequal. Appl., 14, 93-100, 2011) and Karmakar and Sahoo (Results Math., 73, 2018) for a particular class of transcendental meromorphic functions.
\end{abstract}
\section{Introduction}
Throughout this paper, we asume that the readers are familliar with the standard notations of Nevanlinna theory (\cite{8}). Also, we assume that $f$ is a transcendental meromorphic function defined in the complex plane $\mathbb{C}$. It will be convenient to let that $E$ denote any set of positive real numbers of finite linear (Lebesgue) measure, not necessarily same at each occurrence. For any non-constant meromorphic function $f$, we denote by $S(r,f)$ any quantity satisfying $$S(r, f) = o(T(r, f))~~\text{as}~~r\to\infty,~r\not\in E.$$
\begin{defi}
Let $f$ be a non-constant meromorphic function. A meromorphic function $a(z)(\not\equiv 0,\infty)$ is called a \enquote{small function} with respect to $f$ if $T(r,a(z))=S(r,f)$.
\end{defi}
\begin{defi} Let $f$ be non-constant meromorphic function defined in the complex plane $\mathbb{C}$, and $k$ be a positive integer. We say $$M[f]=(f)^{q_{0}}(f')^{q_{1}}\ldots(f^{(k)})^{q_{k}}$$ is a  differential monomial generated by $f$,  where $q_{0},q_{1},\ldots,q_{k}$ are non-negative integers.\par
In this context, the terms  $\mu:=q_{0}+q_{1}+\ldots+q_{k}$ and $\mu_{*}:=q_{1}+2q_{2}+\ldots+kq_{k}$ are known as the degree and weight of the differential monomial respectively.
\end{defi}
\begin{defi}(\cite{f})
Let $a\in \mathbb{C}\cup\{\infty\}$.  For a positive integer  $k$, we denote
\begin{enumerate}
\item [i)] by $N_{k)}\left(r,a;f\right)$ the counting function of $a$-points of $f$ with multiplicity $\leq k$,
\item [ii)] by $N_{(k}\left(r,a;f\right)$ the counting function of $a$-points of $f$ with multiplicity $\geq k$,
\end{enumerate}
Similarly, the reduced counting functions $\ol{N}_{k)}(r,a;f)$ and $\ol{N}_{(k}(r,a;f)$ are defined.
\end{defi}
\begin{defi}(\cite{ld})
For a positive integer $k$, we denote $N_{k}(r,0;f)$ the counting function of zeros of $f$, where a zero of $f$ with multiplicity $q$ is counted $q$ times if $q\leq k$, and is counted $k$ times if  $q> k$.
\end{defi}
In 1959, Hayman proved the following theorem:
\begin{theoA}(\cite{hn})
  If $f$ is a transcendental meromorphic function and $n\geq 3$, then $f^{n}f'$ assumes all finite values except possibly zero infinitely often.
\end{theoA}
Moreover, Hayman (\cite{hn}) conjectured that the Theorem A remains valid for the cases $n = 1,~ 2$. In 1979, Mues (\cite{m}) confirmed the Hayman's Conjecture for $n=2$ and Chen and Fang (\cite{chen}) ensured the conjecture for $n=1$ in 1995.\par
In 1992, Q. Zhang (\cite{qz}) gave the quantitative version of Mues's result  as follows:
\begin{theoB}(\cite{qz}) For a transcendental meromorphic function $f$, the following inequality holds :
$$T(r,f)\leq 6N\bigg(r,\frac{1}{f^{2}f'-1}\bigg)+S(r,f).$$
\end{theoB}
In (\cite{xuyi}), Theorem B was improved by Xu and Yi as
\begin{theoC}(\cite{xuyi}) Let $f$ be a transcendental meromorphic function and $\phi(z)(\not\equiv 0)$ be a small function, then
$$T(r,f)\leq 6\overline{N}\bigg(r,\frac{1}{\phi f^{2}f^{'}-1}\bigg)+S(r,f).$$
\end{theoC}
Also, Huang and Gu (\cite{hg}) extended Theorem B by replacing $f'$ by $f^{(k)}$, where $k(\geq1)$ is an integer.
\begin{theoD}(\cite{hg}) Let $f$ be a transcendental meromorphic function and $k$ be a positive integer. Then
$$T(r,f)\leq 6N\bigg(r,\frac{1}{f^{2}f^{(k)}-1}\bigg)+S(r,f).$$
\end{theoD}
A natural question was raised whether the above inequality still holds if the counting function in Theorem D is replaced by the corresponding reduced
counting function. In this direction, in 2009, Xu, Yi and Zhang (\cite{Xu}) proved the following theorem:
\begin{theoE}
 Let $f$ be a transcendental meromorphic function, and $k(\geq1)$ be a positive integer. If $N_{1}(r,0;f)=S(r,f)$, then
$$T(r,f)\leq 2\ol{N}\bigg(r,\frac{1}{f^{2}f^{(k)}-1}\bigg)+S(r,f).$$
\end{theoE}
Later, in 2011, removing the restrictions on zeros of $f$, Xu, Yi and Zhang (\cite{xu2}) proved the following theorem:
\begin{theoF}
 Let $f$ be a transcendental meromorphic function, and $k(\geq1)$ be a positive integer. Then
$$T(r,f)\leq M\ol{N}\bigg(r,\frac{1}{f^{2}f^{(k)}-1}\bigg)+S(r,f),$$
where $M$ is $6$ if $k=1$, or $k\geq3$ and $M=10$ if $k=2$.
\end{theoF}
Recently, Karmakar and Sahoo(\cite{KS}) further improved the Theorem F and obtained the following result:
\begin{theoG}(\cite{KS})
 Let $f$ be a transcendental meromorphic function, and $n(\geq 2)$, $k(\geq1)$ be any integers, then
$$T(r,f)\leq \frac{6}{2n-3}\ol{N}\bigg(r,\frac{1}{f^{n}f^{(k)}-1}\bigg)+S(r,f).$$
\end{theoG}
From the above discussions the following question is obvious:
\begin{question} Is it possible to replace $f^{n}f^{(k)}$, where $n(\geq 2)$, $k(\geq1)$ be any integers, in the above theorem by $(f)^{q_{0}}(f^{(k)})^{q_{k}}$, where $q_{0}(\geq 2), q_{k}(\geq 1)$ are $k(\geq1)$ non-negative integers?
\end{question}
The aim of this paper is to answer above question by giving some restriction on the poles of $f$.
\section{Main Results}
\begin{theo}\label{th1.1} Let $f$ be a transcendental meromorphic function such that it has no simple pole. Also, let  $q_{0}(\geq 2), q_{k}(\geq 1)$ are $k(\geq1)$ integers. Then
\begin{eqnarray}
\nonumber  T(r,f)&\leq& \frac{6}{2q_{0}-3}\ol{N}\left(r,\frac{1}{(f)^{q_{_{0}}}(f^{(k)})^{q_{_{k}}}-1}\right)+S(r,f).
\end{eqnarray}
\end{theo}
\begin{cor}
  Clearly, Theorem \ref{th1.1} generalies Theorem G for transcendental meromorphic function which has no simple pole.
\end{cor}
\begin{rem}
  Is it possiblie to remove the condition that \enquote{$f$ has no simple pole} when $q_{k}\geq 2$?
\end{rem}
\section{Lemmas}
For a transcendental meromorphic function $f$, we define $$M[f]:=(f)^{q_{_{0}}}(f')^{q_{_{1}}}\ldots(f^{(k)})^{q_{_{k}}}$$ be a differential monomial. In this paper, we assume that $q_{0}(\geq1)$ and $q_{k}(\geq1)$.
\begin{lem}\label{lem1} For a non constant meromorphic function $g$, we obtain
$$N\left(r,\frac{g'}{g}\right)-N\left(r,\frac{g}{g'}\right)=\ol{N}(r,g)+N\left(r,\frac{1}{g}\right)-N\left(r,\frac{1}{g'}\right).$$
\end{lem}
\begin{proof} The proof is same as  the formula (12) of (\cite{jh}).
\end{proof}
\begin{lem}\label{lem3}
Let $f$ be a transcendental meromorphic function and $M[f]$ be a differential monomial in $f$, then $$T\bigg(r,M[f]\bigg)=O(T(r,f))~~\text{and}~~S\bigg(r,M[f]\bigg)=S(r,f).$$
\end{lem}
\begin{proof} The proof is similar to the proof of the Lemma 2.4 of (\cite{ly}).
\end{proof}
\begin{lem}\label{lem1.5}(\cite{f})
Let $f$ be a transcendental  meromorphic function defined in the complex plane $\mathbb{C}$. Then
$$\lim\limits_{r\to\infty}\frac{T(r,f)}{\log r}=\infty.$$
\end{lem}
\begin{lem}\label{lem2} Let $M[f]$ be differential monomial generated by a transcendental meromorphic function $f$. Then $M[f]$ is not identically constant.
\end{lem}
\begin{proof}Here $$\left(\frac{1}{f}\right)^{\mu}=\left(\frac{f'}{f}\right)^{q_1}\left(\frac{f''}{f}\right)^{q_2}\ldots\left(\frac{f^{(k)}}{f}\right)^{q_{k}}\frac{1}{M[f]}.$$
Thus by the first fundamental theorem and lemma of logarithmic derivative, we have
\begin{eqnarray}
\nonumber && \mu T(r,f)\\
\nonumber &\leq& \sum\limits_{i=1}^{k}q_{i}N\left(r,\left(\frac{f^{(i)}}{f}\right)\right)+T(r,M[f])+S(r,f)\\
\nonumber  &\leq& \sum\limits_{i=1}^{k}iq_{i}\left\{\ol{N}(r,0;f)+\ol{N}(r,\infty;f)\right\}+T(r,M[f])+S(r,f)\\
\nonumber  &\leq& \sum\limits_{i=1}^{k}iq_{i}\left\{N(r,0;M[f])+N(r,\infty;M[f])\right\}+T(r,M[f])+S(r,f)\\
\label{eqc1}  &\leq& (2\mu_{\ast}+1)T(r,M[f])+S(r,f),
\end{eqnarray}
Since $f$ is a transcendental meromorphic function, so by Lemma \ref{lem1.5} and inequlity (\ref{eqc1}),  $M[f]$ must be not identically constant.
\end{proof}
\begin{lem}
Let $f$ be a transcendental meromorphic function and $M[f]$ be a differential monomial as given by $M[f] = (f)^{q_0}(f')^{q_1} \cdots (f^{(k)})^{q_k},$ where $q_0( \geq 2), q_1, q_2, \cdots ,q_k~(\geq 1),(k \geq 1)$ are non negative integers. Let $g(z):=M[f]-1$, and $\displaystyle{h(z):=\frac{M'[f]}{f^{q_0 -1}}}$.\par
Next, we define the following function:
\begin{eqnarray}\label{eq0.1}
 F(z)= 2~\bigg( \frac{g'(z)}{g(z)}\bigg)^{2} + 3~\bigg( \frac{g'(z)}{g(z)}\bigg)^{'}-2~\bigg( \frac{h'(z)}{h(z)}\bigg)^{'} \\
\nonumber + \bigg( \frac{h'(z)}{h(z)}\bigg)^{2}-4\bigg( \frac{g'(z)h'(z)}{g(z)h(z)}\bigg),
\end{eqnarray}
Then $F \not\equiv 0.$
\end{lem}
\begin{proof}
On contrary, let us assume that $F \equiv 0.$ Now,
\begin{equation}\label{eq1}
M'[f]=g'=f^{q_0-1}h.
\end{equation}
Thus
\begin{eqnarray}\label{eq2}
N(r,0;f) \leq N(r,0;g').
\end{eqnarray}
\\
\textbf{Claim 1:} First we claim that $g(z) \neq 0.$\\
\textbf{Proof of Claim 1:} \\
If $z_1$ is a zero of $g$ of multiplicity $l~(l \geq 1)$, then $g(z_1)=M[f](z_1) -1 =0$. Thus $f(z_1) \neq 0,\infty$. Now, we consider two cases :\\
\textbf{Case -1.1} $l \geq 2$.
\par
In this case, $z_1$ is a zero of $h$ of order $l-1$. Using Laurent Series expansion of $F$ about $z_1$, one can see that $z_1$ is a pole of $F$ of order $2$ if the coefficient of $(z-z_{1})^{-2}$ in $F$ is non zero, i.e., if $(2l^{2} -3l+(l-1)^{2} +2(l-1)-4l(l-1)) \neq 0$ for all $l$, i.e., the polynomial $-l^{2} +l-1$ has no real solution, which is true by the given condition. Thus $z_1$ is a pole of $F$, which contradicts the fact that $F \equiv 0$. Thus on our assumption $F \equiv 0$, $g(z) \neq 0$.\\
\textbf{Case -1.2} $l=1$.
\par
The equation (\ref{eq1}) yields that  $h(z_1) \neq 0$. In this case, the coefficient of $(z-z_1)^{-2}$ in $F$ is $(-1)$. Thus $z=z_1$ is a pole of $F$ of order $2$, which contradicts the fact that $F \equiv 0$.\\
Hence the claim is true, i.e., $g$ has no zero.\\
\\
\textbf{Claim 2:} Next we claim that $h(z) \neq 0$.\par
\textbf{Proof of Claim 2:}\\
Let $z_2$ be a zero of $h$ of order $m$. Thus from equation (\ref{eq1}), $M'[f](z_2) =0$, i.e., $g'(z_2) =0$. Hence $g(z_2) \neq \infty$. Also, by \textbf{Claim 1}, $g(z_2) \neq 0$. Now, we consider two cases :\\
\textbf{Case -2.1} $m \geq 2$.\par
If $m\geq 2$, then $z_2$ is the zero of $h'(z)$ of order $(m-1)$. So by Laurentz series expansion, one can see that the coefficient of $(z-z_2)^{-2}$ in $F$ is $(m^{2}+2m)$, which is non zero. Thus $z=z_2$ is a pole of $F$ of order 2, which contradicts the fact that $F \equiv 0$.\\
\textbf{Case -2.1} $m=1$.\par
If $m=1$, then the coefficient of $(z-z_2)^{-2}$ in $F$ is $3$, which again  contradicts the fact that $F \equiv 0$.\\
Hence Claim 2 is true.\\
\\
\textbf{Claim 3:} All zeros of $f(z)$ are simple.\\
\textbf{Proof of Claim 3:}\\
If $z_3$ be a zero of $f$ of order $\geq 2$, then by definition of $h$, $h(z_3) =0$, which contradicts the \textbf{Claim 2}. Thus the Claim 3 is true.\\
\\
Now, we define another function as $\phi(z) =\displaystyle{\frac{h(z)}{g(z)}}$. Then
\begin{eqnarray}\label{joy1}
\frac{g'}{g} &=& \phi f^{q_0-1}.
\end{eqnarray}
\begin{eqnarray}\label{joy2}
\frac{h'}{h}&=& \phi f^{q_0-1} + \frac{\phi '}{\phi}.
\end{eqnarray}
Clearly $\phi \not\equiv 0$, otherwise $\displaystyle{\frac{g'}{g} \equiv 0}$, which contradicts Lemma \ref{lem2}.
\\
\\
\textbf{Claim 4:} $\phi(z)$ is an entire function.\\
\textbf{Proof of Claim 4:}\\
As $g$ and $h$ has no zero, so poles of $\phi$ comes from the poles of $h$. Thus poles of $\phi$ comes from the poles of $f$.  Again, zeros of $\phi$ comes from the poles of $g$, i.e., from poles of $f$.
\par
Let $z_4$ be a pole of $f$ of order $t$. Then $z_4$ is a pole of $g$ of order $t\mu + \mu_*$ and pole of $h$ of order $t\mu + \mu_* +1 -t(q_0 -1)$. Thus $z_4$ is a pole of $\phi$ of order $1-t(q_0 -1)$ if $1-t(q_0 -1) > 0$ and $z_4$ is a zero of $\phi$ of order $t(q_0 -1)-1$ if $t(q_0 -1)-1 >0$.
\par
As $q_0 \geq 2$, so $\phi$ is an \textbf{entire function}.  Also, if $q_0 =2$, then zeros of $\phi$ occur only at multiple poles of $f$ and if $q_0 >2$, then zeros of $\phi$ occur at only poles of $f$.
\par
Next, in view of Lemma \ref{lem2}, we can write
\begin{equation}\label{gul2}
\frac{1}{f^{\mu}}=\frac{g+1}{f^{\mu}} - \frac{g'}{f^{\mu}}\frac{g}{g'}.
\end{equation}
and
\begin{eqnarray}\label{gul3}
 \frac{\mu}{q_0 -1}T(r,\phi)= \frac{\mu}{q_0 -1}m(r,\phi)&=& \frac{\mu}{q_0 -1}m\bigg(r,\frac{g'}{g}\frac{1}{f^{q_0-1}}\bigg)\\
\nonumber &\leq & m\bigg(r,\frac{1}{f^{\mu}}\bigg)+S(r,f).
\end{eqnarray}
Thus using Lemma \ref{lem1}, equation (\ref{gul2}) and inequality (\ref{gul3}), (\ref{eq2}), we have
\begin{eqnarray}
\nonumber \frac{\mu}{q_0 -1}T(r,\phi)&=& \frac{\mu}{q_0 -1}m(r,\phi)\\
\nonumber & \leq & m\bigg(r,\frac{g}{g'}\bigg)+S(r,f)\\
\nonumber & \leq & N\bigg(r,\frac{g'}{g}\bigg) -N\bigg(r,\frac{g}{g'}\bigg) +S(r,f) \\
\nonumber & \leq &\overline{N}(r,g) +N\bigg(r,\frac{1}{g}\bigg) - N\bigg(r,\frac{1}{g'}\bigg)+S(r,f) \\
\nonumber & \leq & \overline{N}(r,f) -N\bigg(r,\frac{1}{f}\bigg)+S(r,f)
\end{eqnarray}
Again, using (\ref{gul2}), we have
\begin{eqnarray}
\nonumber && \mu\cdot m\bigg(r,\frac{1}{f}\bigg)  \leq  \overline{N}(r,f)-N\bigg(r,\frac{1}{f}\bigg)+S(r,f)\\
\nonumber \text{i.e.,} &&\mu T(r,f) \leq  N(r,f) +(\mu -1)N\bigg(r,\frac{1}{f}\bigg)+S(r,f)\\
\nonumber \text{i.e.,} && m(r,f) +(\mu -1)m\bigg(r,\frac{1}{f}\bigg) \leq S(r,f).
\end{eqnarray}
Hence,
\begin{eqnarray}\label{eq5}
m(r,f) &=& m\bigg(r,\frac{1}{f}\bigg) = S(r,f), \\
\label{eq6}
T(r,\phi) &=& S(r,f).
\end{eqnarray}
Next, we consider two cases :\\
\textbf{Case-1} Assume $q_0 > 2$.\\
If $z_4$ is  a pole of $f$ of order $t$, then $z_4$ is a zero of $\phi$ of order $t(q_0 -1) - 1$. As $t(q_0 -1) -1 \geq 2t -1 \geq t$, so
\begin{equation}\label{eq7}
N(r,f) \leq N(r,0,\phi)
\end{equation}
Combining (\ref{eq5}),(\ref{eq6}) and (\ref{eq7}), we get
\begin{equation}
T(r,f)=S(r,f),
\end{equation}
which is absurd as $f$ is a non constant transcendental meromorphic function. Thus our assumption is wrong. Hence $F \not\equiv 0$.\\
\textbf{Case-2} Next, we assume that $q_0 =2$.\\
Substituting (\ref{joy1}) and (\ref{joy2}) in (\ref{eq0.1}) and using the fact that $F \equiv 0$, we obtain
\begin{eqnarray}
f^{2} \phi ^{2} +\bigg[2  \bigg(\frac{\phi '}{\phi}\bigg)' - \bigg(\frac{\phi '}{\phi}\bigg)^{2}\bigg]+f\phi' - f' \phi \equiv 0.
\end{eqnarray}
From Lemma (\ref{lem2}), it is clear that $\phi\not\equiv0$. If $z_5$ is the zero of $f$, then $\phi (z_5) \neq 0,\infty$.  Thus proceeding similarly as in Lemma 3 of (\cite{hg}), we can write
\begin{equation}\label{su1}
f^{(i)}(z_5) = \frac{l_{i1}(z_5)}{\phi (z_5)},
\end{equation}
where $l_{i1}(z)$ are the differential monomials in $\displaystyle{\frac{\phi '}{\phi}}$ for $ i=1,2, \cdots, k$.
\par
Since $g(z_5)=-1$ and $h(z_5)=q_{0}\bigg((f')^{q_1 +1}(f'')^{q_2} \cdots (f^{(k)})^{q_k}\bigg)(z_5)$, so,
\begin{equation}\label{eq9}
\phi (z_5) = -q_0\bigg((f')^{q_1 +1}(f'')^{q_2} \cdots (f^{(k)})^{q_k}\bigg)(z_5).
\end{equation}
Thus using (\ref{eq9}) and (\ref{su1}), we have
\begin{equation}
\phi (z_5) = -q_0\bigg(\frac{(l_{11})^{q_1 +1}(l_{21})^{q_2} \cdots (l_{k1})^{q_k}}{(\phi)^{q_1 +q_2+ \cdots +q_k +1}}\bigg)(z_5).
\end{equation}
Next we define $$\displaystyle{G :=\phi^{q_1 +q_2+ \cdots +q_k +2}  + q_0~ (l_{11})^{q_1 +1}(l_{21})^{q_2} \cdots (l_{k1})^{q_k}}.$$
If $G \not\equiv 0$, then $$N(r,0;f)=\overline{N}(r,0;f)\leq N(r,0;G) \leq O(T(r,\phi)) +O(1) =S(r,f).$$
Thus $T(r,f) = T(r,\frac{1}{f})+O(1)=S(r,f)$, a contradiction as $f$ is non constant transcendental meromorphic function.\\
If $G \equiv 0$, then $$\displaystyle{\phi^{q_1 +q_2+ \cdots +q_k +2}  = - q_0~ (l_{11})^{q_1 +1}(l_{21})^{q_2} \cdots (l_{k1})^{q_k}}.$$
Thus by Lemma of Logarithmic derivative, $T(r,\phi) = m(r,\phi) =S(r,\phi)$, i.e., $\phi$ is a polynomial or a constant (as $\phi$ is an entire function). \par
Now, proceeding similarly as in Lemma 3 of (\cite{hg}), one can show that $f$ is rational, which is impossible. Hence the proof.
\end{proof}
\begin{lem}(\cite{ld})\label{lda}
  Let $f$ be a transcendental meromorphic function and $\alpha=\alpha(z)(\not\equiv 0,\infty)$ be a small function of $f$. If  $\psi=\alpha(f)^{n}(f^{(k)})^{p}$, where $n(\geq 0)$, $p(\geq 1)$, $k(\geq 1)$ are integers, then for any small function $a=a(z)(\not\equiv 0,\infty)$ of $\psi$,
$$(p + n)T(r, f)\leq \overline{N}(r,\infty; f) + \overline{N}(r, 0; f) + pN_{k}(r, 0; f) + \overline{N}(r, a; \psi) + S(r, f).$$
\end{lem}
\section {Proof of the Theorem}
\begin{proof} [\textbf{Proof of Theorem \ref{th1.1}}] We define
$$g(z):= (f)^{q_0}(f^{(k)})^{q_k}-1,$$ where $q_0( \geq 2), q_k~(\geq 1) ~~k (\geq 1)$ are non negative integers, and $\displaystyle{h(z):=\frac{g'}{f^{q_0 -1}}}$.
Also,
\begin{eqnarray}\label{cor1}
\nonumber F(z)= 2~\bigg( \frac{g'(z)}{g(z)}\bigg)^{2} + 3~\bigg( \frac{g'(z)}{g(z)}\bigg)^{'}-2~\bigg( \frac{h'(z)}{h(z)}\bigg)^{'} \\
\label{eq0}+ \bigg( \frac{h'(z)}{h(z)}\bigg)^{2}-4\bigg( \frac{g'(z)h'(z)}{g(z)h(z)}\bigg),
\end{eqnarray}
Clearly $F(z)\not\equiv 0,$ and $f$ has no simple poles. Next we define another function as
\begin{eqnarray*}
\beta&:=&q_{0}(f')(f^{(k)})^{q_k}+q_{k}f(f^{(k)})^{q_{k}-1}f^{(k+1)}-f(f^{(k)})^{q_{k}}\frac{g'}{g}.
\end{eqnarray*}
Then \begin{equation}\label{fr1}
       f^{q_{0}-1}\beta=-\frac{g'}{g},
     \end{equation}
and
\begin{equation}\label{fr2}
  h=-\beta g,
\end{equation}
and
\begin{equation}\label{fr3}
  \beta^{2}F=\beta^{2}\{(\frac{g'}{g})'-(\frac{g'}{g})^{2}\}-2\beta\beta'(\frac{g'}{g})+(\beta')^{2}-2(\beta\beta'-(\beta')^{2})
\end{equation}
We note that
\begin{enumerate}
  \item [i)] Equation (\ref{fr2}) gives that the zeros of $h$ come from the zeros of $\beta$ or, the zeros of $g$.
\item [ii)] Equation (\ref{fr1}) gives that the multiple poles of $f$ with multiplicity $p(\geq 2)$ are the zeros of $\beta$ with multiplicity $(q_{0}-1)p-1$.
\item [iii)] If $z_{0}$ is a zero of $g$, then it can not be a pole of $f$. Thus from equation (\ref{fr1}), it is clear that $z_{0}$ is a simple pole of $\beta$.
\item [iv)] From (iii) and equation (\ref{fr3}) gives that the poles of  $ \beta^{2}F$ only come from the zeros of $g$. Moreover, poles of  $ \beta^{2}F$ have multiplicity atmost $4$. Thus
\begin{equation}\label{gul30}
  N(r,\beta^{2}F)\leq 4\ol{N}(r,\frac{1}{g}).
\end{equation}
Since $m(r,F)=S(r,f)$ and $m(r,\beta)=S(r,f)$, therefore $m(r,\beta^{2}F)=S(r,f)$. Thus
\begin{equation}\label{gul31}
 T(r,\beta^{2}F)\leq 4\ol{N}(r,\frac{1}{g})+S(r,f).
\end{equation}
\end{enumerate}
Let $z_{0}$ be a zero of $f$ of multiplicity $q(\geq k+1)$. Then equation (\ref{fr1}) gives that it zero of $\beta$ of order atleast $q_{k}(q-k)+(q-1)$. Therefore it is a zero of $\beta^{2}F$ of order atleast $2(q_{k}(q-k)+(q-1))-2=(2q-2)+2q_{k}(q-k)-2\geq (2q-2)$. Thus
\begin{equation}\label{block1}
2N(r,\frac{1}{f})-2\ol{N}(r,\frac{1}{f})\leq N(r,\frac{1}{\beta^{2}F})\leq T(r,\beta^{2}F)+O(1)\leq 4\ol{N}(r,\frac{1}{g})+S(r,f).
\end{equation}
Again, from Lemma \ref{lda}, we have
\begin{equation}\label{block2}
(q_{0} +q_{k})T(r, f)\leq \overline{N}(r,\infty; f) + \overline{N}(r, 0; f) + q_{k}N_{k}(r, 0; f) + \overline{N}(r,\frac{1}{g}) + S(r, f).
\end{equation}
Combinig twice of (\ref{block2}) and (\ref{block1}), we obtain
\begin{eqnarray*}
2(q_{0} +q_{k})T(r, f)+2N(r,\frac{1}{f})-2\ol{N}(r,\frac{1}{f}) &\leq& 2\overline{N}(r,\infty; f) + 2\overline{N}(r, 0; f) + 2q_{k}N_{k}(r, 0; f)\\
&& + 6\overline{N}(r,\frac{1}{g}) + S(r, f).
\end{eqnarray*}
Since $f$ has no simple pole, so we have
\begin{eqnarray}
2(q_{0} +q_{k})T(r, f) &\leq& 2\overline{N}_{(2}(r,\infty; f) + 2\overline{N}(r, 0; f) + 2q_{k}N_{k}(r, 0; f)\\
\nonumber && + 6\overline{N}(r,\frac{1}{g}) + S(r, f).
\end{eqnarray}
i.e.,
\begin{eqnarray}
&&(2q_{0}-3)T(r, f)+m(r,f)+N(r,f)+(2+2q_{k})m(r,\frac{1}{f})+(2+2q_{k})N(r,\frac{1}{f})\\
\nonumber &\leq& 2\overline{N}_{(2}(r,\infty; f) + 2\overline{N}(r, 0; f) + 2q_{k}N_{k}(r, 0; f)+ 6\overline{N}(r,\frac{1}{g}) + S(r, f).
\end{eqnarray}
i.e.,
\begin{eqnarray}
&&(2q_{0}-3)T(r, f)+N(r,f)+(2+2q_{k})N(r,\frac{1}{f})\\
\nonumber &\leq& 2\overline{N}_{(2}(r,\infty; f) + 2\overline{N}(r, 0; f) + 2q_{k}N_{k}(r, 0; f)+ 6\overline{N}(r,\frac{1}{g}) + S(r, f).
\end{eqnarray}
i.e.,
\begin{eqnarray}
&&(2q_{0}-3)T(r, f)+N(r,f)\\
\nonumber &\leq& 2\overline{N}_{(2}(r,\infty; f) + 6\overline{N}(r,\frac{1}{g}) + S(r, f).
\end{eqnarray}
Thus
\begin{eqnarray}
&&(2q_{0}-3)T(r, f)\\
\nonumber &\leq&  6\overline{N}(r,\frac{1}{g}) + S(r, f).
\end{eqnarray}
This completes the proof.
\end{proof}
\begin{center} {\bf Acknowledgement} \end{center}
The author is grateful to the anonymous referee for his/her valuable suggestions which considerably improved the presentation of the paper.\par
The research work of the first and the fourth authors are  supported by the Department of Higher Education, Science and Technology \text{\&} Biotechnology, Govt.of West Bengal under the sanction order no. 216(sanc) /ST/P/S\text{\&}T/16G-14/2018 dated 19/02/2019.\par
The second and the third authors are thankful to the Council of Scientific and Industrial Research, HRDG, India for granting Junior Research
Fellowship during the tenure of which this work was done. 

\end{document}